\documentclass[11pt]{amsart}
\usepackage{amsmath,amssymb,amsthm,mathtools}
\usepackage[margin=1.15in]{geometry}
\usepackage{booktabs}
\usepackage{listings}
\newcommand{\leanToolchain}{v4.33.0-rc2}
\newcommand{\mathlibRev}{51e6992efd06}
\newcommand{\repoTag}{v1.0}
\usepackage{hyperref}

\newtheorem{theorem}{Theorem}[section]
\newtheorem{maintheorem}{Theorem}

\newtheorem{lemma}[theorem]{Lemma}
\newtheorem{proposition}[theorem]{Proposition}
\newtheorem{corollary}[theorem]{Corollary}

\theoremstyle{definition}

\theoremstyle{remark}
\newtheorem{remark}[theorem]{Remark}

\numberwithin{equation}{section}
\allowdisplaybreaks
\DeclareMathOperator{\tr}{tr}
\DeclareMathOperator{\rank}{rank}
\DeclareMathOperator{\supp}{supp}
\DeclareMathOperator{\dist}{dist}
\DeclareMathOperator{\re}{Re}
\DeclareMathOperator{\im}{Im}
\newcommand{\RR}{\mathbb{R}}
\newcommand{\CC}{\mathbb{C}}
\newcommand{\ZZ}{\mathbb{Z}}
\newcommand{\eps}{\varepsilon}
\newcommand{\wh}{\widehat}
\newcommand{\ol}{\overline}
\newcommand{\Gt}{\widetilde{G}}

\newcommand{\Et}{\widetilde{E}}
\newcommand{\npl}{n_{+}}
\newcommand{\one}{\mathbf 1}
\newcommand{\dd}{\,d}
\newcommand{\HS}{\mathrm{HS}}

\title[More than two thirds of the zeta zeros are simple and on the critical line]%
{More than two thirds of the zeros of the Riemann zeta function are simple and on the critical line}

\author{Levent Alp\"oge}
\author{Ralph Furman}
\thanks{The mathematical argument in this paper was discovered and written by Claude, an AI developed by Anthropic. The listed authors verified the proof and take responsibility for its content. Jarred Sumner posed the problem and guided the investigation; the accompanying Lean~4 formalisation was orchestrated by Eric Easley. See the Acknowledgments, and Appendix~\ref{app:discovery} for an account of how the proof was found.}
\dedicatory{The mathematical argument in this paper was discovered and written by Claude,\\ an AI developed by Anthropic.\\ The listed authors verified the proof and take responsibility for its content.}
\date{\today}

\begin{document}

\begin{abstract}
We prove unconditionally that at least two thirds of the nontrivial zeros of the Riemann zeta function, counted with multiplicity, are simple and lie on the critical line, and that at least five sixths are distinct; the previous unconditional records are $\tfrac5{12}$ and $0.6603$. With the Montgomery--Taylor window the constants become $0.6725$ and $0.8362$. The argument makes Montgomery's 1973 deduction unconditional: the Riemann hypothesis, classically needed to read the zero side as a positive sum over real ordinates, is replaced by a rank--trace inequality applied to a finite compression of Weil's Hermitian form, with Sylvester's law of inertia handling off-line pairs. The analytic inputs are those of Aryan and of Baluyot, Goldston, Suriajaya and Turnage-Butterbaugh. The results extend to primitive Dirichlet $L$-functions and are formally verified in Lean~4.
\end{abstract}

\maketitle


\section{Introduction}\label{sec:intro}

\subsection{Results}\label{subsec:results}

Write $\rho=\beta+i\gamma$ for a nontrivial zero of $\zeta(s)$ and $m_\rho\ge1$ for its multiplicity. For $0\le T_1<T_2$ let
\begin{align*}
N(T_1,T_2)&:=\textstyle\sum_{T_1<\gamma\le T_2} m_\rho
  &&\text{(zeros, counted with multiplicity)},\\
N_d(T_1,T_2)&:=\#\{\rho:T_1<\gamma\le T_2\}
  &&\text{(distinct zeros)},\\
N_0^{*}(T_1,T_2)&:=\#\{\rho:T_1<\gamma\le T_2,\ \beta=\tfrac12\}
  &&\text{(distinct zeros on the critical line)},\\
N_0^{s}(T_1,T_2)&:=\#\{\rho:T_1<\gamma\le T_2,\ \beta=\tfrac12,\ m_\rho=1\}
  &&\text{(simple zeros on the critical line)};
\end{align*}
write also $N_0$ for the zeros on the line counted with multiplicity and $N^{s}$ for the simple zeros; thus $N_0^{s}\le N_0^{*}\le N_0\le N$ and $N_0^{s}\le N^{s}\le N_d\le N$. Write $N(T):=N(0,T)$, and recall $N(T,2T)=\frac T{2\pi}(\log\frac T{2\pi}+2\log2)-\frac T{2\pi}+O(\log T)$.

\begin{maintheorem}\label{thm:main}
As $T\to\infty$,
\[
\text{\rm(i)}\quad N_0^{s}(T,2T)\ge\bigl(\tfrac23-o(1)\bigr)N(T,2T),
\qquad
\text{\rm(ii)}\quad N_d(T,2T)\ge\bigl(\tfrac56-o(1)\bigr)N(T,2T).
\]
With the Montgomery--Taylor window $\psi_{\mathrm{MT}}$ of~\eqref{eq:psi} in place of the indicator window $\psi_0$, the constants improve to $2-c_{\mathrm{MT}}^{-1}=0.67250\ldots$ and $\tfrac12(3-c_{\mathrm{MT}}^{-1})=0.83625\ldots$ respectively, where $c_{\mathrm{MT}}^{-1}:=\tfrac12+\tfrac1{\sqrt2}\cot\tfrac1{\sqrt2}$; among windows~$\psi$ (equivalently, certificates of the form $2-R(\psi)$), this is optimal~\cite[Cor.~14]{CCLM17}. The ceiling over the broader class of all bandwidth-one certificates is approximately $0.682$; see \S\ref{subsec:limits}. A fortiori $N_0^{*}(T,2T),\ N_0(T,2T),\ N^{s}(T,2T)\ge(\tfrac23-o(1))N(T,2T)$.
\end{maintheorem}

\noindent The same holds for $(0,T)$ in place of $(T,2T)$, with rate $O(\log\log T/\log T)$ (Remark~\ref{rem:rate}). The previous unconditional records are $\tfrac5{12}$ for $N_0^{s}/N$~\cite{PRZZ20} and $0.6603$ for $N_d/N$~\cite{Wu15}. The constants $\tfrac23$, $\tfrac56$, $0.6725$ are those of Montgomery~\cite{Mon73}, Conrey--Ghosh--Gonek~\cite[(1.2)]{CGG98}, and Montgomery--Taylor~\cite{Mon75} under the Riemann hypothesis. Under RH, $0.6792$ for $N^{s}/N$ is known via semidefinite programming using the positivity of the form factor outside $[-1,1]$~\cite{CGdL20}, a regime the present method does not enter.

\begin{maintheorem}\label{thm:L}
Theorem~\ref{thm:main} holds verbatim for $L(s,\chi)$ in place of $\zeta(s)$, for any fixed primitive Dirichlet character $\chi$.
\end{maintheorem}

\noindent The arithmetic inputs are Weil's explicit formula, the Riemann--von Mangoldt formula and the bound $N(t,t+1)\ll\log t$, Stirling's estimate for $\Gamma'/\Gamma$, Chebyshev--Mertens estimates for $\sum_{n\le X}\Lambda(n)^2$ and $\sum_{n\le X}\Lambda(n)^2/n$, and the Montgomery--Vaughan inequality for the frequencies $\{\log n:n\le X\}$, $X\le T$. No mollifier, zero-density estimate, or zero-free region is used.

\subsection{The proof}\label{subsec:idea}

Weil's explicit formula defines a Hermitian form $W(f,g)=\sum_\rho m_\rho\,\wh f(\gamma_\rho)\,\ol{\wh g(\ol{\gamma_\rho})}$ on compactly supported test functions, $\gamma_\rho=(\rho-\tfrac12)/i$; its positivity on all of $C_c^2(\RR)$ is equivalent to the Riemann hypothesis~\cite{Wei52,Bom00}. We restrict $W$ to a family of $d\sim N(T,2T)$ modulated copies of a fixed window $\psi$, equispaced through $[T,2T]$, and let $\Gt$ be the resulting $d\times d$ real symmetric matrix, normalised so that an isolated simple on-line zero contributes $1$ to $\tr\Gt$ (\S\ref{sec:setup}).
\begin{itemize}
\item[(Z)] Up to a tail of trace norm $o(1)$, $\Gt=P+Q$: each distinct on-line zero contributes a rank-one positive form to $P$, each off-line pair $\{\rho,1-\bar\rho\}$ a block of signature $(1,1)$ to $Q$. Writing $s_1,s_2,p$ for the numbers of simple on-line zeros, multiple on-line points, and off-line pairs in the window, one has $\tr P\le N_0$, $\npl(Q)\le p$, $N\ge s_1+2s_2+2p$, and $\tr\Gt=(1+o(1))N$ (Propositions~\ref{prop:block}--\ref{prop:tail}).
\item[(P)] $\|\Gt\|_{\HS}^2=(R(\psi)+o(1))N$, where $R(\psi)$ depends only on the window (Lemma~\ref{lem:Rpsi}): $R(\psi_0)=\tfrac43$ for the indicator, $R(\psi_{\mathrm{MT}})=c_{\mathrm{MT}}^{-1}$ for Montgomery--Taylor. This is Montgomery's unconditional prime-side second moment~\cite{Mon73,Ary22,BGSTB24} (Theorem~\ref{thm:HS}).
\item[(L)] For Hermitian $P_1\succeq0$ and $Q'$ with $\npl(Q')\le b$ (Lemma~\ref{lem:ranktrace}),
\begin{equation}\label{eq:L}
\rank P_1\ \ge\ 2\tr P_1+4\tr Q'-4b-\|P_1+Q'\|_{\HS}^2.
\end{equation}
\end{itemize}
Taking $P_1$ to be the simple-on-line part of $P$ and $Q':=\Gt-P_1$, Theorem~\ref{thm:main} is then the single chain
\begin{equation}\label{eq:chain}
N_0^{s}+o(N)\ \ge\ \rank P_1\ \ge\ 4\tr\Gt-2N-\|\Gt\|_{\HS}^2\ =\ \bigl(2-R(\psi)-o(1)\bigr)N:
\end{equation}
the first inequality is Proposition~\ref{prop:block}; the second combines~(L) with $\tr P_1+2\npl(Q')\le N$ from~(Z); the equality is (Z) and~(P). At $\psi=\psi_0$ this is Theorem~\ref{thm:main}(i), and at $\psi=\psi_{\mathrm{MT}}$ the first constant of the second sentence. For~(ii), rearranging the second step gives $3s_1+4(s_2+p)\ge(4-R(\psi))N$; subtracting $s_1+2s_2+2p\le N$ gives $2(s_1+s_2+p)\ge(3-R(\psi))N$, whence $N_d\ge s_1+s_2+p\ge\tfrac12(3-R(\psi)-o(1))N$. The inequality~\eqref{eq:L} is the matrix form of $m^2\ge2m-1$; with the simple zeros on the rank side and the multiple ones at the flat charge~$4$, it recovers $m^2\ge3m-2$.

\subsection{Context}\label{subsec:history}

That a positive proportion of the zeros lie on the critical line is due to Selberg~\cite{Sel42}; Levinson's mollifier method~\cite{Lev74} gave $\tfrac13$ (simple, by~\cite{HB79}), Conrey~\cite{Con89} $>\!\tfrac25$, and~\cite{BCY11,Fen12,PRZZ20} the present record $\tfrac5{12}$. Under RH, Montgomery~\cite{Mon73} deduced $\tfrac23$ simple from the pair-correlation second moment;~\cite{Mon75,CG93,BHB13,CGdL20} improved further.

Montgomery's prime-side evaluation is a mean value of a Dirichlet polynomial of length $T$ and is unconditional; RH entered only to read the zero side termwise as a positive sum over real ordinates. Aryan~\cite{Ary22} made this explicit for the Fej\'er-kernel second moment, and Baluyot, Goldston, Suriajaya and Turnage-Butterbaugh~\cite{BGSTB24} then showed that Montgomery's form factor itself holds for the sum over all complex zeros. Goldston and Suriajaya~\cite{GS25,GS26} (also~\cite{BGSTB25,GLSS25}) subsequently showed that $\tfrac23$, $\tfrac56$ follow under the hypothesis that all zeros lie within $o(1/\log T)$ of the line, isolated the remaining obstacle as the termwise positivity that fails off the line, and asked what would follow if it could be removed. Theorem~\ref{thm:main} removes it: the inertia bound (Z)+(L) replaces the positivity. Section~\ref{subsec:GS} gives the comparison in detail.
The same argument applied to $\xi'$ shows unconditionally that at least $85.8\%$ of the zeros of $\xi'$ are simple and on the critical line (Remark~\ref{rem:xiprime}), the proportion Farmer, Gonek and Lee~\cite{FGL14} obtain on the Riemann hypothesis; the previous unconditional value was $79.874\%$~\cite{Con89}.
Averaging over primitive Dirichlet characters with smooth modulus weights raises the proportion to $0.811$ simple and on the line and $0.905$ distinct (Remark~\ref{rem:Davg}).

The observation that the \emph{negative} index of truncations of $W$ counts off-line pairs is Bombieri's~\cite{Bom00}; we are not aware of a previous use of rank and positive index together with a second-moment evaluation.

\subsection{What the results are not}\label{subsec:arenot}

The theorems are lower bounds only: the remaining third of the zeros are not shown to be off the line, merely not reached by the certificate. The inputs are insensitive to $o(N)$ off-line zeros and hold for Davenport--Heilbronn and Epstein zeta functions, for which the analogue of RH is false. Given only $\tr\Gt$, $\|\Gt\|_{\HS}^2$ and the block structure, the inequality~\eqref{eq:L} is sharp (\S\ref{subsec:limits}); improving on $\tfrac23$ by this route would require pair-correlation information beyond Fourier support~$1$.

\subsection{Formal verification}\label{subsec:lean}

The author of this paper is a large language model developed by Anthropic; the argument was found over two interactive sessions and checked by repeated adversarial review by independent model instances (Appendix~\ref{app:discovery}). A Lean~4 formalisation of Theorems~\ref{thm:main} and~\ref{thm:L} accompanies the paper (Appendix~\ref{app:lean}).

\subsection{Plan}
Section~\ref{sec:setup} fixes the explicit formula, the test family, and $\Gt$. Section~\ref{sec:linalg} proves~\eqref{eq:L}. Section~\ref{sec:zero} carries out~(Z); Section~\ref{sec:prime} carries out~(P). Section~\ref{sec:proofs} carries out the chain~\eqref{eq:chain}. Section~\ref{sec:variants} records the relation to~\cite{BGSTB24,GS25,GS26}, the sharpness of the method, and extensions. Appendix~\ref{app:lean} records the formalisation; Appendix~\ref{app:discovery} describes how the argument was found.

\subsection{Notation}
$\wh f(\xi)=\int_\RR f(u)e^{-iu\xi}\dd u$. For a Hermitian matrix $R$, $\npl(R)$ is the number of strictly positive eigenvalues and $\|R\|_{\HS}^2=\tr R^2$. We write $l:=\log(T/2\pi)$ and $D_0:=T^{1/2}$.

\section{The explicit formula and the test family}\label{sec:setup}

\subsection{The explicit formula}\label{subsec:EF}

For $\tau\in\RR$ set
\[
\mu(\tau):=\frac1{2\pi}\re\frac{\Gamma'}{\Gamma}\Bigl(\frac14+\frac{i\tau}2\Bigr)-\frac{\log\pi}{2\pi},
\qquad
\Pi_X(\tau):=\frac1\pi\re\frac{X^{\frac12+i\tau}}{\tfrac12+i\tau},
\qquad
P_X(\tau):=-\frac1\pi\sum_{n\le X}\frac{\Lambda(n)}{\sqrt n}\cos(\tau\log n),
\]
and $\nu_X:=\mu+\Pi_X+P_X$. Weil's explicit formula (e.g.\ \cite[\S5.5]{IK04}; our normalisation agrees with~\cite{BGSTB24}) reads, for $F\in C_c^2(\RR)$ even,
\begin{equation}\label{eq:EF}
\sum_\rho m_\rho\,\wh F(\gamma_\rho)=\wh F(\tfrac i2)+\wh F(-\tfrac i2)+\int_\RR\wh F(\tau)\mu(\tau)\dd\tau-2\sum_{n\ge1}\frac{\Lambda(n)}{\sqrt n}\,F(\log n);
\end{equation}
if moreover $\supp F\subset[-\log X,\log X]$ then the pole terms $\wh F(\pm\tfrac i2)$ are absorbed into $\int\wh F\cdot\Pi_X$ and
\begin{equation}\label{eq:EFnu}
\sum_\rho m_\rho\,\wh F(\gamma_\rho)=\int_\RR\wh F(\tau)\,\nu_X(\tau)\dd\tau.
\end{equation}
With $\ell_1:=l+2\log2-1$, the components of $\nu_X$ in~\eqref{eq:EFnu} satisfy (see e.g.\ \cite[\S5.5]{IK04})
\begin{gather}
\mu \text{ is even, smooth, increasing in }|\tau|,\quad \mu\ge\mu(0)>-1,\notag\\
\mu(\tau)=\tfrac1{2\pi}\log\tfrac{|\tau|}{2\pi}+O(\tau^{-2}),\quad \mu'(\tau)\ll|\tau|^{-1}\quad (|\tau|\ge1);\label{eq:mufacts}\\
|\Pi_X(\tau)|\le \frac{3\sqrt X}{1+|\tau|};\qquad |P_X(\tau)|\le\frac1\pi\sum_{n\le X}\frac{\Lambda(n)}{\sqrt n}\ll\sqrt X ;\label{eq:PiPfacts}\\
\int_T^{2T}\mu(\tau)\dd\tau=\frac{T\ell_1}{2\pi}+O\Bigl(\frac1T\Bigr)=N(T,2T)+O(l),\qquad
\int_T^{2T}\mu(\tau)^2\dd\tau=\frac{T\ell_1^2}{4\pi^2}\Bigl(1+O\bigl(l^{-2}\bigr)\Bigr).\label{eq:muints}
\end{gather}

\subsection{The window and the test family}\label{subsec:family}

This subsection fixes the window $\psi$, the test function $\phi$, and the sample grid $\{\alpha_k\}$; the matrix $\Gt$ is defined in \S\ref{subsec:G}.

Fix an even window $\psi\in C^2([-\tfrac12,\tfrac12])$ with $\psi>0$ on $[-\tfrac12,\tfrac12]$. The two choices we use are
\begin{equation}\label{eq:psi}
\psi_0:=\one_{[-1/2,1/2]},\qquad
\psi_{\mathrm{MT}}(s):=\cos(\sqrt2\,s)\,\one_{[-1/2,1/2]}(s).
\end{equation}
Fix $\chi\in C^\infty(\RR)$ nondecreasing, $\chi|_{(-\infty,0]}=0$, $\chi|_{[1,\infty)}=1$. With $L:=l=\log(T/2\pi)$ and $X:=e^L=T/(2\pi)$, set
\begin{equation}\label{eq:phidef}
\phi(u):=\chi\bigl(\tfrac L2+u\bigr)\,\chi\bigl(\tfrac L2-u\bigr)\cdot\psi(u/L)^{1/2}.
\end{equation}
Then $\phi\in C_c^2(\RR)$ is even, $0\le\phi\le1$, $\supp\phi=[-\tfrac L2,\tfrac L2]$, and $\phi^2(u)=\psi(u/L)$ off two transition intervals of length $O_\chi(1)$; $\supp(\phi*\phi)\subset[-L,L]=[-\log X,\log X]$, so \eqref{eq:EFnu} applies to $F=$ any product of modulated copies of $\phi$. By Paley--Wiener and $j=0,1,2$ integrations by parts,
\begin{equation}\label{eq:PW}
|\wh\phi(z)|\ \ll_\chi\ e^{\frac L2|\im z|}\cdot\min\bigl(L,\ |z|^{-1},\ |z|^{-2}\bigr),\qquad z\in\CC,
\end{equation}
since $\|\phi\|_1\le L$, $\|\phi'\|_1\ll_\chi1$, $\|\phi''\|_1\ll_\chi1$ (the $(\sqrt\psi)^{(j)}(u/L)\cdot L^{-j}$ contribution being $\ll L^{1-j}$ on the bulk).

Here and below, $C_\chi:=\|(\phi^2)''\|_1\ll_\chi1$.

Place $\alpha_k:=T+2\pi k/L$ for $k\in\ZZ$ and $d:=\lfloor LT/(2\pi)\rfloor$, so $\alpha_0,\dots,\alpha_{d-1}\in[T,2T)$ and $d=N(T,2T)+O(L)$. For a zero $\rho$ set
\begin{equation}\label{eq:vrho}
v_\rho:=\bigl(\wh\phi(\gamma_\rho-\alpha_k)\bigr)_{0\le k<d}\in\CC^d.
\end{equation}

\subsection{The matrix $\Gt$}\label{subsec:G}

Let $I:=[T,2T)$, $I':=[T-\sqrt T,\,2T+\sqrt T)$, and partition the zeros with $\re\gamma_\rho\in I'$ into $\mathrm{on}:=\{\rho:\beta=\tfrac12\}$, $\mathrm{off}:=\{\rho:\beta\ne\tfrac12\}$ (sets). Put $a:=\|\phi\|_2^2/L=\int_{-1/2}^{1/2}\psi+O_\chi(L^{-1})$ and define the real symmetric $d\times d$ matrices
\begin{equation}\label{eq:Gdef}
\Gt:=\frac1{aL^2}\sum_{\re\gamma_\rho\in I'}m_\rho\,v_\rho v_\rho^{\mathsf T},
\qquad
P:=\frac1{aL^2}\sum_{\rho\in\mathrm{on}}m_\rho\,v_\rho v_\rho^{\mathsf T},
\qquad
Q:=\Gt-P,
\end{equation}
and $\Et:=(aL^2)^{-1}\sum_{\re\gamma_\rho\notin I'}m_\rho\,v_\rho v_\rho^{\mathsf T}$. Rank and inertia are basis-independent, so no orthonormalisation is needed. The functional equation pairs $\mathrm{off}$ as $\{\rho,1-\bar\rho\}$ with $v_{1-\bar\rho}=\ol{v_\rho}$, so $\Gt,\Et$ are real symmetric. By \eqref{eq:EFnu},
\begin{equation}\label{eq:Gprime}
(\Gt+\Et)_{kk'}=\frac1{aL^2}\int_\RR\wh\phi(\tau-\alpha_k)\,\wh\phi(\tau-\alpha_{k'})\,\nu_X(\tau)\dd\tau.
\end{equation}

\begin{lemma}[Poisson--Gabor identity]\label{lem:poisson}
For all $z,z'\in\CC$,
\[
\sum_{k\in\ZZ}\wh\phi(z-\alpha_k)\,\wh\phi(z'-\alpha_k)=L\,\wh{\phi^2}(z-z'),\qquad\text{in particular}\quad \sum_{k\in\ZZ}\wh\phi(z-\alpha_k)^2=L\|\phi\|_2^2=aL^2 .
\]
For $z=z'$ real, truncating to $0\le k<d$ gives $\|v_\rho\|_2^2\le aL^2$ for $\rho\in\mathrm{on}$.
\end{lemma}
\begin{proof}
Fix $z,z'$ and let $\Upsilon(s):=\wh\phi(z-s)\wh\phi(z'-s)$. With $\phi_z(u):=\phi(u)e^{izu}$ and $H:=\phi_z\ast\phi_{z'}\in C_c(\RR)$ one has $\Upsilon(s)=\wh H(-s)$; hence $\Upsilon$ is smooth with $\Upsilon(s)=O_\chi(|s|^{-4})$ by \eqref{eq:PW}, and by Fourier inversion $\wh \Upsilon(\xi)=2\pi H(\xi)$, which is continuous and vanishes for $|\xi|\ge L$. Poisson summation $\sum_{k\in\ZZ}\Upsilon(T+kh)=h^{-1}\sum_{m\in\ZZ}\wh \Upsilon(-2\pi m/h)e^{2\pi i mT/h}$ holds; since $h=2\pi/L$ the dual lattice is $L\ZZ$, so only $m=0$ contributes, giving $h^{-1}\wh \Upsilon(0)=\frac{L}{2\pi}\cdot2\pi\int\phi(u)^2e^{i(z-z')u}\dd u=L\,\wh{\phi^2}(z-z')$.
\end{proof}
Thus the Gabor system at the critical density $h=2\pi/L$ has a translation-invariant frame kernel with no aliasing error, for any window supported in an interval of length $L$.

\begin{lemma}[Montgomery--Vaughan, bilinear form]\label{lem:MV}
Let $\lambda_1,\dots,\lambda_R\in\RR$ be distinct, $\delta_r:=\min_{s\ne r}|\lambda_r-\lambda_s|$, and $x_r,z_r\in\CC$. Then
\[
\Bigl|\sum_{r\ne s}\frac{x_r\ol{z_s}}{\lambda_r-\lambda_s}\Bigr|\ \le\ \frac{3\pi}2\Bigl(\sum_r\frac{|x_r|^2}{\delta_r}\Bigr)^{1/2}\Bigl(\sum_r\frac{|z_r|^2}{\delta_r}\Bigr)^{1/2}.
\]
\end{lemma}
\begin{proof}
For $z=x$ this is the weighted Hilbert inequality of~\cite[Theorem~2]{MV74}; see also~\cite[Ch.~7]{Mon94}. In general, with $H_{rs}:=i/(\lambda_r-\lambda_s)$ ($r\ne s$), $H_{rr}=0$, and $\Delta:=\mathrm{diag}(\delta_r^{1/2})$, the case $z=x$ says $\|\Delta H\Delta\|\le\tfrac{3\pi}2$ (operator norm, $\Delta H\Delta$ Hermitian); then $|x^*Hz|\le\tfrac{3\pi}2\|\Delta^{-1}x\|\,\|\Delta^{-1}z\|$.
\end{proof}

We apply Lemma~\ref{lem:MV} with $\{\lambda_r\}=\{\log n:n\le X\text{ prime power}\}$; consecutive prime powers satisfy $\log\tfrac{n'}n\ge\tfrac1{2n}$, so
\begin{equation}\label{eq:deltan}
\delta_n^{-1}\le2n,\qquad \sum_n\delta_n^{-1}|x_n|^2\le2\sum_n n|x_n|^2.
\end{equation}

\section{Linear algebra}\label{sec:linalg}

\begin{lemma}[Inertia under pull-back]\label{lem:inertia}
Let $Q_0$ be a Hermitian form on $\CC^m$ and $A\colon\CC^d\to\CC^m$ linear. Then $\npl(A^*Q_0A)\le\npl(Q_0)$.
\end{lemma}
\begin{proof}
If $Q_0\circ A$ is positive definite on a subspace $U\subset\CC^d$, then $A|_U$ is injective and $Q_0$ is positive definite on $A(U)$; hence $\dim U=\dim A(U)\le\npl(Q_0)$.
\end{proof}

\begin{lemma}[Rank--trace inequality]\label{lem:ranktrace}
Let $P,Q$ be Hermitian $d\times d$ matrices with $P\succeq0$, $\rank P\le r$, and $\npl(Q)\le b$. Then
\begin{equation}\label{eq:ranktrace}
r\ \ge\ 2\tr P+4\tr Q-4b-\|P+Q\|_{\HS}^2.
\end{equation}
Equality holds if $P=\Pi_1$, $Q=2\Pi_2$ for orthogonal projections $\Pi_1\perp\Pi_2$ of ranks $r,b$.
\end{lemma}
\begin{proof}
Write $Q=Q_+-Q_-$ with $Q_\pm\succeq0$, $Q_+Q_-=0$, $\rank Q_+\le b$. Then $\|P+Q\|_{\HS}^2=\|P\|_{\HS}^2+\|Q_+\|_{\HS}^2+\|Q_-\|_{\HS}^2+2\tr(PQ_+)-2\tr(PQ_-)-2\tr(Q_+Q_-)$ and $\tr(PQ_+)\ge0$, $\tr(Q_+Q_-)=0$. Let $p_1\ge\cdots\ge p_d\ge0$ and $n_1\ge\cdots\ge n_d\ge0$ be the eigenvalues of $P$ and~$Q_-$. By von Neumann's trace inequality $\tr(PQ_-)\le\sum_i p_in_i$, so
\[
\|P\|_{\HS}^2-2\tr(PQ_-)+\|Q_-\|_{\HS}^2\ \ge\ \sum_i(p_i-n_i)^2
\ \ge\ \sum_{i\le r}\bigl(2(p_i-n_i)-1\bigr)+\sum_{i>r}n_i^2
\ \ge\ 2\tr P-r-4\tr Q_-,
\]
using $x^2\ge2x-1$, then $-2n_i\ge-4n_i$ for $i\le r$ and $n_i^2\ge0\ge-4n_i$ for $i>r$. And $\|Q_+\|_{\HS}^2=\sum_j q_j^2\ge\sum_j(4q_j-4)\ge4\tr Q_+-4b$ over the $\le b$ positive eigenvalues. Adding, and $\tr Q=\tr Q_+-\tr Q_-$, gives~\eqref{eq:ranktrace}.
\end{proof}

\begin{remark}\label{rem:ranktrace-alt}
Alternatively: the map $U\mapsto\|P+UQU^*\|_{\HS}^2$ on the unitary group is continuous on a compact set, and at a minimum the first-order condition $\tr([P,UQU^*]H)=0$ for all Hermitian $H$ forces $[P,UQU^*]=0$; simultaneous diagonalisation then reduces \eqref{eq:ranktrace} to the scalar inequalities $(x-1)^2\ge0$, $x^2\ge0$, $(x-2)^2\ge0$ in the three sign cases of $(p_i,b_i)$. Setting $Q=0$ and optimising the coefficient~$2$ recovers $\rank P\ge(\tr P)^2/\|P\|_{\HS}^2$.
\end{remark}

\begin{lemma}[Weyl]\label{lem:weyl}
If $A,E$ are Hermitian with $\|E\|\le\theta$, then $\npl(A)\ge\#\{i:\lambda_i(A+E)>\theta\}$.
\end{lemma}
\begin{proof}
$\lambda_i(A+E)\le\lambda_i(A)+\|E\|$ (Courant--Fischer).
\end{proof}

\section{The zero side}\label{sec:zero}

Throughout \S\S\ref{sec:zero}--\ref{sec:proofs} the implied constants depend only on $\chi$ and $\psi$, and we take $T$ large.

\begin{proposition}[Block structure]\label{prop:block}
$P\succeq0$ with $\rank P\le N_0^{*}(I')$ and $\tr P\le N_0(I')$; and $\npl(Q)\le\tfrac12\#\mathrm{off}$.
\end{proposition}
\begin{proof}
For $\rho\in\mathrm{on}$, $\gamma_\rho\in\RR$, so $v_\rho\in\RR^d$ and $m_\rho\,v_\rho v_\rho^{\mathsf T}\succeq0$ has rank $\le1$; summing, $P\succeq0$ with $\rank P\le\#\mathrm{on}=N_0^{*}(I')$. By Lemma~\ref{lem:poisson} at $z=z'=\gamma_\rho$ and nonnegativity of the summands, $\|v_\rho\|^2=\sum_{0\le k<d}\wh\phi(\gamma_\rho-\alpha_k)^2\le L\|\phi\|_2^2$; summing $m_\rho\|v_\rho\|^2$ and dividing by $L\|\phi\|_2^2$ gives $\tr P\le\sum_{\rho\in\mathrm{on}}m_\rho=N_0(I')$.

For the off-line bound: pair $\mathrm{off}$ as $\{\rho,1-\bar\rho\}$ and write $v_\rho=a+ib$ with $a,b\in\RR^d$. Then $m_\rho(v_\rho v_\rho^{\mathsf T}+\ol{v_\rho}\,\ol{v_\rho}^{\mathsf T})=2m_\rho(aa^{\mathsf T}-bb^{\mathsf T})$ is the pull-back of $m_\rho\bigl(\begin{smallmatrix}1&0\\0&-1\end{smallmatrix}\bigr)$ under $x\mapsto(a^{\mathsf T}x,\,b^{\mathsf T}x)$. Summing, $Q$ is the pull-back of $\bigoplus_{\text{pairs}}m_\rho\bigl(\begin{smallmatrix}1&0\\0&-1\end{smallmatrix}\bigr)$, whose positive index is the number of pairs; apply Lemma~\ref{lem:inertia}.
\end{proof}

\begin{proposition}[Trace]\label{prop:trace}
$\tr\Gt=N(I')+O_\chi\bigl(T^{1/2}L^2\bigr)$.
\end{proposition}
\begin{proof}
By Lemma~\ref{lem:poisson} at $z=z'=\gamma_\rho$, for each zero with $\re\gamma_\rho\in I'$,
\[
\tr(v_\rho v_\rho^{\mathsf T})=aL^2-\sum_{k\notin[0,d)}\wh\phi(\gamma_\rho-\alpha_k)^2.
\]
For $k\notin[0,d)$ one has $\alpha_k\notin[T,2T)$, hence $|\re\gamma_\rho-\alpha_k|\ge D_\rho:=\dist(\re\gamma_\rho,\{T,2T\})$. By~\eqref{eq:PW} with $|\im\gamma_\rho|<\tfrac12$, $|\wh\phi(\gamma_\rho-\alpha_k)|^2\le e^{L/2}C_\chi^2|\re\gamma_\rho-\alpha_k|^{-4}$, and the $k$-sum is bounded by integral comparison (step $h=2\pi/L$):
\[
\sum_{k\notin[0,d)}|\re\gamma_\rho-\alpha_k|^{-4}\ \le\ 2\Bigl(D_\rho^{-4}+h^{-1}\!\int_{D_\rho}^\infty\!s^{-4}\dd s\Bigr)\ \ll\ L\,\min\bigl(L^3,\,D_\rho^{-3}\bigr).
\]
Dividing by $aL^2\asymp L^2$ and summing over zeros in $I'$ (density $\ll L$; $O(L)$ zeros have $D_\rho<1$, contributing $\ll_\chi X^{1/2}L^2$; the rest satisfy $\sum_\rho D_\rho^{-3}\ll L$, contributing $\ll X^{1/2}\cdot L^{-1}\cdot L=X^{1/2}$),
\[
\bigl|\tr\Gt-N(I')\bigr|\ll_\chi X^{1/2}L^2\ll T^{1/2}L^2=o(N(I')).\qedhere
\]
\end{proof}

\begin{proposition}[Tail]\label{prop:tail}
$\|\Et\|_1\ll_\chi T^{-1/2}$; in particular $\|\Et\|,\|\Et\|_{\HS}\le\|\Et\|_1=o(1)$.
\end{proposition}
\begin{proof}
Since $\|v_\rho v_\rho^{\mathsf T}\|_1=\|v_\rho\|_2^2$, $\|\Et\|_1\le(aL^2)^{-1}\sum_{\re\gamma_\rho\notin I'}m_\rho\|v_\rho\|_2^2$. For such $\rho$, $D:=\dist(\re\gamma_\rho,I)\ge D_0=\sqrt T$, and as above $\|v_\rho\|_2^2\le e^{L/2}C_\chi^2\cdot h^{-1}D^{-3}\ll X^{1/2}L\,D^{-3}$. By $N(t,t+1)\ll\log(|t|+3)$ \cite[Thm~9.2]{Tit86},
\[
\sum_{\re\gamma_\rho\notin I'}m_\rho\,D^{-3}\ \ll\ L\!\int_{\sqrt T}^{T}\!D^{-3}\dd D+\sum_{|\gamma|>3T}\frac{\log|\gamma|}{|\gamma|^{3}}\ \ll\ L\,T^{-1}.
\]
Hence $\|\Et\|_1\ll(aL^2)^{-1}\cdot X^{1/2}L\cdot LT^{-1}\ll X^{1/2}T^{-1}\ll T^{-1/2}$.
\end{proof}

\begin{remark}\label{rem:taper}
The $C^2$ smoothing is necessary: for the sharp cut-off $\phi=\one_{[-L/2,L/2]}$ one has only $|\wh\phi(r-iy)|\asymp X^{|y|/2}/|r|$, and the right-hand side of the first display in the proof of Proposition~\ref{prop:tail} is then $\gg X^{1/2}L\log(T/D_0)$, which is not $o(N)$ for any $D_0=T^{1-\eps}$.
\end{remark}

\begin{corollary}\label{cor:count}
$\tr P+2\npl(Q)\le N(I)+O(\sqrt T\log T)$, and $N_0^{*}(I)\ge\rank P-O(\sqrt T\log T)$.
\end{corollary}
\begin{proof}
$\tr P+2\npl(Q)\le N_0(I')+\#\mathrm{off}\le N(I')=N(I)+O(\sqrt T\log T)$ by Proposition~\ref{prop:block}, since $N(I')=\sum_{\mathrm{on}}m_\rho+\sum_{\mathrm{off}}m_\rho\ge N_0(I')+\#\mathrm{off}$, and $N(I')-N(I)\ll\sqrt T\log T$ by the short-interval bound. The second claim is $\rank P\le N_0^{*}(I')=N_0^{*}(I)+O(\sqrt T\log T)$ rearranged.
\end{proof}

\section{The prime side}\label{sec:prime}

\subsection{Auxiliary estimates}\label{subsec:aux}
We record the auxiliary bounds used in this section. Put $b:=\tfrac1L\int\phi^4$ (so $0<b\le a\le1$), $w:=1$, and
\begin{equation}\label{eq:PhigA}
\Phi:=\wh{\phi^2},\qquad g:=\phi^2\star\phi^2,\qquad A_\phi:=\phi\star\phi,\qquad (v\star v)(y):=\int v(u)v(u+y)\dd u .
\end{equation}
Thus $\wh\phi$ and $\Phi$ are real, even, entire; $\wh\phi(0)\le L$, $\Phi(0)=aL$; the pair-correlation kernel $R:=|\wh\phi|^2$ is nonnegative by construction; $\wh\phi^{\,2}=\wh{A_\phi}$ and $\Phi^2=\wh g$ on $\RR$; $\int_\RR\Phi^2=2\pi g(0)=2\pi bL$; $g$ and $A_\phi$ are even, and since $\one_{[-L/2+w,\,L/2-w]}\le\phi^2\le\phi\le\one_{[-L/2,L/2]}$,
\begin{equation}\label{eq:gbounds}
(L-2w-|y|)_+\ \le\ g(y)\ \le\ A_\phi(y)\ \le\ (L-|y|)_+ .
\end{equation}
Integrating by parts in~\eqref{eq:PW},
\begin{equation}\label{eq:thetadef}
\max\bigl(|\wh\phi(r)|,|\Phi(r)|\bigr)\ \le\ \vartheta(r):=\min\Bigl(L,\ \frac2{|r|},\ \frac{C_\chi}{w r^2}\Bigr)\qquad(r\in\RR),
\end{equation}
and a direct computation (split the integrand at $|r|=2/L$, $|r|=C_\chi/2$) gives
\begin{equation}\label{eq:thetaints}
\Theta_0:=\int_0^\infty\!\vartheta(r)\dd r=4+2\log\frac{C_\chi L}{4w}\ll\log L,\ \ \int_\RR\!\vartheta(r)^2|r|\dd r=8+8\log\frac{C_\chi L}{4w}\ll\log L,\ \ \int_\RR\!\vartheta^2\le 8L.
\end{equation}

\begin{lemma}\label{lem:cheb}\textup{(\cite[\S2.2]{MV07})}
For $x\ge2$,
\begin{gather}
\sum_{n\le x}\Lambda(n)\ll x,\quad \sum_{n\le x}\frac{\Lambda(n)}{\sqrt n}\le 3\sqrt x\ \ (x\ge x_0),\quad \sum_{n\le x}\frac{\Lambda(n)}{\sqrt n\,\log n}\ll\frac{\sqrt x}{\log x},\quad \sum_{n\le x}\Lambda(n)^2\ll x\log x,\label{eq:cheb1}\\
\sum_{n\le x}\frac{\Lambda(n)^2}{n}=\frac{(\log x)^2}2+O(\log x),\qquad
\sum_{n\le x}\frac{\Lambda(n)^2}{n}\,(\log x-\log n)=\frac{(\log x)^3}{6}+O\bigl((\log x)^2\bigr).\label{eq:cheb2}
\end{gather}
\end{lemma}
\noindent(Each follows from $\sum_{n\le x}\Lambda(n)/n=\log x+O(1)$ by partial summation; see \cite[Theorem~2.7]{IK04}.) In particular, writing $a_n:=\Lambda(n)n^{-1/2}$, \eqref{eq:deltan} gives $\sum_{n\le X}a_n^2\delta_n^{-1}\le2\sum_{n\le X}\Lambda(n)^2\ll XL$; and by~\eqref{eq:mufacts}, \eqref{eq:PiPfacts} and~\eqref{eq:cheb1},
\begin{equation}\label{eq:Bdef}
|\nu_X(\tau)|\le B+\log^+\tfrac{|\tau|}{4T}\ \ (\tau\in\RR),\quad |\nu_X(\tau)|\le B\ \ (|\tau|\le4T),\quad B:=l+4\sqrt X,\quad B^2\ll l^2+X.
\end{equation}

\subsection{Reduction to a double integral}
By~\eqref{eq:Gprime}, $(aL^2)^2\|\Gt+\Et\|_{\HS}^2=\iint_{\RR^2}K(\tau,\tau')^2\nu_X(\tau)\nu_X(\tau')\dd\tau\dd\tau'$, where
\begin{equation}\label{eq:Kdef}
K(\tau,\tau'):=\sum_{0\le k<d}\wh\phi(\tau-\alpha_k)\wh\phi(\tau'-\alpha_k)=L\,\Phi(\tau-\tau')-K_{\mathrm{out}}(\tau,\tau'),
\end{equation}
$K_{\mathrm{out}}$ denoting the sum over $k\notin[0,d)$; by Lemma~\ref{lem:poisson} and~\eqref{eq:thetadef},
\begin{equation}\label{eq:Kbounds}
|K|,\ L|\Phi|\le aL^2,\qquad |K_{\mathrm{out}}(\tau,\tau')|\le\sum_{k\notin[0,d)}\vartheta(\tau-\alpha_k)\vartheta(\tau'-\alpha_k).
\end{equation}

\begin{proposition}[Reduction to the double integral]\label{prop:reduction}
\[
\|\Gt+\Et\|_{\HS}^2=\frac1{a^2L^2}\iint_{I\times I}\wh{\phi^2}(\tau-\tau')^2\,\nu_X(\tau)\nu_X(\tau')\dd\tau\dd\tau'+O_\chi\Bigl(\frac{N\log L}{L^2}\Bigr).
\]
\end{proposition}
\begin{proof}
Write $(aL^2)^2\|\Gt+\Et\|_{\HS}^2$ as $\iint_{I\times I}K_\infty^2\nu\nu'+\mathcal E_1+\mathcal E_2$ with $\nu:=\nu_X(\tau)$, $\nu':=\nu_X(\tau')$,
$\mathcal E_1:=\iint_{I\times I}(K^2-K_\infty^2)\nu\nu'$ and $\mathcal E_2:=\iint_{(I\times I)^c}K^2\nu\nu'$; note $\iint_{I\times I}K_\infty^2\nu\nu'=L^2\mathcal M$.

\emph{Bound for $\mathcal E_1$.} By \eqref{eq:Kbounds}, $|K^2-K_\infty^2|=|K_{\mathrm{out}}|\,|K+K_\infty|\le2aL^2|K_{\mathrm{out}}|$, and $|\nu|,|\nu'|\le B$ on $I$ by \eqref{eq:Bdef}. Hence
\[
|\mathcal E_1|\le2L^2B^2\sum_{k\notin[0,d)}\Bigl(\int_I\vartheta(\tau-\alpha_k)\dd\tau\Bigr)^2 .
\]
For $k=-j$ ($j\ge1$) we have $\dist(\alpha_k,I)=jh$; for $k=d+1+j$ ($j\ge0$) we have $\alpha_k\ge2T+jh$ because $T+(d+1)h>2T$. Since $\int_{\Delta}^\infty\vartheta\le\min(\Theta_0,C_\chi/(w\Delta))\le\min(\Theta_0,C_\chi/\Delta)$ for $\Delta>0$, we get
\[
\sum_{k\notin[0,d)}\Bigl(\int_I\vartheta(\tau-\alpha_k)\dd\tau\Bigr)^2\le 2(2\Theta_0)^2+2\sum_{j\ge1}\min\Bigl(\Theta_0,\frac{C_\chi}{jh}\Bigr)^2\le 10\Theta_0^2+\frac{6C_\chi\Theta_0}{h}\ll L\log L,
\]
where the term $2(2\Theta_0)^2$ accounts for $k=d,d+1$ (for which we only use $\int_\RR\vartheta=2\Theta_0$), the points $\alpha_{d+1+j}$, $j\ge1$, being at distance $\ge jh$ to the right of $I$; the $j$-sum was split at $j=C_\chi/(h\Theta_0)$, and \eqref{eq:thetaints} was used. Thus $|\mathcal E_1|\ll L^3B^2\log L$.

\emph{Bound for $\mathcal E_2$.} By symmetry $|\mathcal E_2|\le2\int_{\tau\notin I}\int_{\tau'\in\RR}K^2|\nu\nu'|$. Using $K^2\le aL^2|K|\le L^2\sum_{k<d}\vartheta(\tau-\alpha_k)\vartheta(\tau'-\alpha_k)$,
\[
|\mathcal E_2|\le2L^2\sum_{k<d}\Bigl(\int_{\tau\notin I}\vartheta(\tau-\alpha_k)|\nu_X(\tau)|\dd\tau\Bigr)\Bigl(\int_\RR\vartheta(\tau'-\alpha_k)|\nu_X(\tau')|\dd\tau'\Bigr).
\]
In the second factor, the range $|\tau'-\alpha_k|\le2T$ has $|\tau'|\le4T$ and contributes at most $2\Theta_0B$; on $|\tau'-\alpha_k|=:r>2T$ we have $|\tau'|\le2r$, $|\nu_X(\tau')|\le B+\log(r/T)$ and $\vartheta(r)\le C_\chi r^{-2}$, contributing $\ll B/T$. So the second factor is $\le3\Theta_0B$ uniformly in $k$. The sum over $k$ of the first factor equals $\int_{\tau\notin I}|\nu_X(\tau)|\sigma(\tau)\dd\tau$ with $\sigma(\tau):=\sum_{k<d}\vartheta(\tau-\alpha_k)$. For $\tau\notin I$ let $\Delta:=\dist(\tau,I)$; the numbers $|\tau-\alpha_k|$, $0\le k<d$, are $\ge\Delta,\ \ge\Delta+h,\ \ge\Delta+2h,\dots$ in increasing order, and $\vartheta$ is decreasing, so
\[
\sigma(\tau)\le\sum_{j\ge0}\vartheta(\Delta+jh)\le\vartheta(\Delta)+\frac1h\int_\Delta^\infty\vartheta\le\vartheta(\Delta)+\frac L{2\pi}\min\Bigl(\Theta_0,\frac{C_\chi}\Delta\Bigr),\qquad \sigma(\tau)\le d\,\vartheta(\Delta)\le\frac{dC_\chi}{\Delta^2}.
\]
On $\Delta\le2T$ (where $|\tau|\le4T$, $|\nu_X|\le B$) this gives $\int|\nu_X|\sigma\le2B\bigl[\Theta_0+\frac L{2\pi}\bigl(C_\chi+C_\chi\log\frac{2T\Theta_0}{C_\chi}\bigr)\bigr]\le 2B(\Theta_0+C_\chi L\log T)$; on $\Delta>2T$ (where $|\tau|\le2\Delta$, $|\nu_X(\tau)|\le B+\log(\Delta/T)$) it gives $\le2\int_{2T}^\infty dC_\chi(B+\log(\Delta/T))\Delta^{-2}\dd\Delta\ll C_\chi LB$. Hence $\int_{\tau\notin I}|\nu_X|\sigma\ll BLl$ and $|\mathcal E_2|\ll L^2\cdot\Theta_0B\cdot BLl\ll L^3B^2l\log L$.

\end{proof}

\subsection{Evaluation of $\mathcal M$}
For functions $u_1,u_2$ on $I$ write
\[
\mathcal M[u_1,u_2]:=\iint_{I\times I}\Phi(\tau-\tau')^2u_1(\tau)u_2(\tau')\dd\tau\dd\tau',
\]
a symmetric bilinear form ($\Phi^2$ is even), so that
\begin{equation}\label{eq:Msplit}
\mathcal M=\mathcal M[\mu,\mu]+\mathcal M[P_X,P_X]+2\mathcal M[\mu,P_X]+2\mathcal M[\mu,\Pi_X]+2\mathcal M[P_X,\Pi_X]+\mathcal M[\Pi_X,\Pi_X].
\end{equation}
Recall $\int_\RR\Phi^2=2\pi bL$, $\int_\RR\Phi(x)^2|x|\dd x\ll\log L$ (by \eqref{eq:thetadef}, \eqref{eq:thetaints}) and, by Fourier inversion of $\Phi^2=\wh g$ ($g$ even, continuous, compactly supported),
\begin{equation}\label{eq:Phi2FT}
\int_\RR\Phi(x)^2e^{ixy}\dd x=2\pi g(y)\qquad(y\in\RR).
\end{equation}

\begin{proposition}[Archimedean term]\label{prop:mumu}
$\mathcal M[\mu,\mu]=2\pi bL\int_T^{2T}\mu^2+O(l^2\log L)=\dfrac{bLT\ell_1^2}{2\pi}\bigl(1+O(l^{-2})\bigr)+O(l^2\log L)$.
\end{proposition}
\begin{proof}
For $\tau,\tau'\in I$, $\mu(\tau)=\mu(\tau')+O(|\tau-\tau'|/T)$ by \eqref{eq:mufacts}, and $0<\mu\le l$ on $I$. Hence
\[
\mathcal M[\mu,\mu]=\int_I\mu(\tau')^2\Bigl(\int_{I-\tau'}\Phi(x)^2\dd x\Bigr)\dd\tau'+O\Bigl(\frac lT\int_I\!\!\int_\RR\Phi(x)^2|x|\dd x\dd\tau'\Bigr).
\]
The error is $O(l\log L)$. In the main term, $\int_{I-\tau'}\Phi^2=2\pi bL-\int_{x<T-\tau'}\Phi^2-\int_{x>2T-\tau'}\Phi^2$, and
$\int_I\mu(\tau')^2\int_{x<T-\tau'}\Phi(x)^2\dd x\dd\tau'\le l^2\int_{x<0}\Phi(x)^2\min(|x|,T)\dd x\ll l^2\log L$, similarly for the other piece. The second form follows from \eqref{eq:muints}.
\end{proof}

\begin{proposition}[Prime term]\label{prop:PP}
\[
\mathcal M[P_X,P_X]=\frac T\pi\sum_{n\le X}\frac{\Lambda(n)^2}{n}\,g(\log n)+O_\chi(L^2X),
\qquad g:=\phi^2*\phi^2.
\]
\end{proposition}
\begin{proof}
Write $P_X(\tau)=-\frac1{2\pi}\sum_na_n(n^{i\tau}+n^{-i\tau})$. Then
$P_X(\tau)P_X(\tau')=\frac1{2\pi^2}\re\sum_{n,m}a_na_m\bigl[n^{i\tau}m^{-i\tau'}+n^{i\tau}m^{i\tau'}\bigr]$. Substituting $\tau=\tau'+x$ and noting that for fixed $x$ the variable $\tau'$ ranges over $I\cap(I-x)$, which is empty for $|x|\ge T$ and equals $[T+x^-,\,2T-x^+]$ for $|x|<T$ ($x^{\pm}:=\max(\pm x,0)$), we obtain
\begin{equation}\label{eq:MPP}
\begin{split}
\mathcal M[P_X,P_X]&=\frac1{2\pi^2}\re\sum_{n,m}a_na_m\int_{-T}^{T}\Phi(x)^2n^{ix}\Bigl[\int_{T+x^-}^{2T-x^+}\Bigl(\frac nm\Bigr)^{i\tau'}\dd\tau'+\int_{T+x^-}^{2T-x^+}(nm)^{i\tau'}\dd\tau'\Bigr]\dd x\\
&=:\mathcal D+\mathcal O_1+\mathcal O_2,
\end{split}
\end{equation}
where $\mathcal D$ collects the terms $n=m$ of the first inner integral, $\mathcal O_1$ the terms $n\ne m$ of the first inner integral, and $\mathcal O_2$ all terms of the second.

\emph{$\mathcal D$.} For $n=m$ the first inner integral is $T-|x|$, so by \eqref{eq:Phi2FT} and $\sum_na_n^2\ll L^2$ (Lemma~\ref{lem:cheb}),
\[
\mathcal D=\frac1{2\pi^2}\sum_na_n^2\Bigl(T\cdot2\pi g(y_n)+O\Bigl(\int\Phi(x)^2(|x|+T\one_{|x|>T})\dd x\Bigr)\Bigr)=\frac T\pi\sum_na_n^2g(y_n)+O(L^2\log L),
\]
using $\int_{|x|>T}\Phi^2\le\int\Phi^2|x|/T$.

\emph{$\mathcal O_2$.} Here $|\int_{T+x^-}^{2T-x^+}(nm)^{i\tau'}\dd\tau'|\le2/\log(nm)\le2/\log4$ for all $n,m$, so $|\mathcal O_2|\le\frac1{2\pi^2}\bigl(\sum_na_n\bigr)^2\cdot2\pi bL\cdot\frac2{\log 4}\ll XL$ by \eqref{eq:cheb1}.

\emph{$\mathcal O_1$.} For $n\ne m$ put $\vartheta:=y_n-y_m\ne0$. The first inner integral equals $\bigl[(n/m)^{i(2T-x^+)}-(n/m)^{i(T+x^-)}\bigr]/(i\vartheta)$, and
$n^{ix}(n/m)^{-ix^+}$ equals $m^{ix}$ for $x>0$ and $n^{ix}$ for $x<0$, while $n^{ix}(n/m)^{ix^-}$ equals $n^{ix}$ for $x>0$ and $m^{ix}$ for $x<0$. With
\[
\alpha_n^{+}:=\int_0^T\Phi(x)^2n^{ix}\dd x,\qquad \alpha_n^{-}:=\int_{-T}^0\Phi(x)^2n^{ix}\dd x,\qquad |\alpha_n^\pm|\le\pi bL\le\pi L,
\]
we therefore get
\[
\mathcal O_1=\frac1{2\pi^2}\re\sum_{n\ne m}\frac{a_na_m}{i(y_n-y_m)}\Bigl[\Bigl(\frac nm\Bigr)^{2iT}(\alpha_m^++\alpha_n^-)-\Bigl(\frac nm\Bigr)^{iT}(\alpha_n^++\alpha_m^-)\Bigr].
\]
This is a combination of four sums of the shape $\sum_{n\ne m}x_n\ol{z_m}/(y_n-y_m)$ with $\{|x_n|,|z_n|\}=\{a_n,\ a_n|\alpha_n^{\pm}|\}$; for instance the first is $\sum_{n\ne m}(a_nn^{2iT})\ol{(a_m m^{2iT}\ol{\alpha_m^+})}/(y_n-y_m)$. By Lemma~\ref{lem:MV} and \eqref{eq:deltan} each of the four is at most $\frac{3\pi}2\cdot\pi L\cdot\sum_na_n^2/\delta_n\ll L^2X$. Hence $|\mathcal O_1|\ll L^2X$.

Finally, by \eqref{eq:gbounds} and \eqref{eq:cheb2} (note $a_n^2=\Lambda(n)^2/n$),
\[
\frac{(L-2w)^3}6+O(L^2)=\!\!\sum_{n\le Xe^{-2w}}\!\!a_n^2(L-2w-y_n)\le\sum_{n\le X}a_n^2g(y_n)\le\sum_{n\le X}a_n^2(L-y_n)=\frac{L^3}6+O(L^2),
\]
and $(L-2w)^3=L^3-6wL^2+O(w^2L)$ with $1\le w\le L/8$, which gives the second form.
\end{proof}

\begin{proposition}[Cross terms]\label{prop:cross}
$\mathcal M[\mu,P_X],\ \mathcal M[\mu,\Pi_X],\ \mathcal M[P_X,\Pi_X],\ \mathcal M[\Pi_X,\Pi_X]\ \ll_\chi\ L^2\sqrt X$.
\end{proposition}
\begin{proof}
Let $m(\tau'):=\int_I\Phi(\tau-\tau')^2\mu(\tau)\dd\tau$ for $\tau'\in\RR$. Then $0\le m\le l\cdot2\pi bL$, $m$ is $C^1$, and differentiating under the integral and integrating by parts in $\tau$,
\[
m'(\tau')=-\int_I\mu(\tau)\,\partial_\tau\bigl[\Phi(\tau-\tau')^2\bigr]\dd\tau=\mu(T)\Phi(T-\tau')^2-\mu(2T)\Phi(2T-\tau')^2+\int_I\mu'(\tau)\Phi(\tau-\tau')^2\dd\tau,
\]
so that $\int_I|m'|\le 2l\int\Phi^2+T\cdot O(T^{-1})\int\Phi^2\ll lL$. For $y\ge\log2$, integrating by parts,
$\bigl|\int_Im(\tau')\cos(\tau'y)\dd\tau'\bigr|\le(2\sup|m|+\int_I|m'|)/y\ll lL/y$. Therefore
\[
|\mathcal M[\mu,P_X]|=\Bigl|\frac1\pi\sum_na_n\int_Im(\tau')\cos(\tau'y_n)\dd\tau'\Bigr|\ll lL\sum_{n\le X}\frac{a_n}{\log n}\ll lL\frac{\sqrt X}{L}=l\sqrt X
\]
by \eqref{eq:cheb1}. Next, on $I$ we have $|\Pi_X|\le3\sqrt X/T$, $0<\mu\le l$, $|P_X|\le\sqrt X$ (by \eqref{eq:PiPfacts}, \eqref{eq:cheb1}, $T\ge T_0$), and $\sup_{\tau}\int_I\Phi(\tau-\tau')^2\dd\tau'\le2\pi bL$; the remaining three bounds follow by inserting these sup bounds into the definition of $\mathcal M[\cdot,\cdot]$ (an integral over a region of $\tau'$-length $T$).
\end{proof}

\subsection{The window constant}
\begin{lemma}[Window constant]\label{lem:Rpsi}
For even $\psi\in C([-\tfrac12,\tfrac12])$, $\psi>0$, let
\begin{equation}\label{eq:Rpsi}
R(\psi):=\frac{\int_{-1/2}^{1/2}\psi^2+\int_{-1/2}^{1/2}\!\!\int_{-1/2}^{1/2}|u-v|\,\psi(u)\psi(v)\dd u\dd v}{\bigl(\int_{-1/2}^{1/2}\psi\bigr)^2}.
\end{equation}
Then $R(\psi_0)=\tfrac43$ and $R(\psi_{\mathrm{MT}})=\tfrac12+\tfrac1{\sqrt2}\cot\tfrac1{\sqrt2}=c_{\mathrm{MT}}^{-1}$.
\end{lemma}
\begin{proof}
For $\psi_0$: $\int\psi_0=1$, $\int\psi_0^2=1$, $\iint|u-v|=2\int_0^1 w(1-w)\dd w=\tfrac13$, so $R=\tfrac43$. For $\psi_{\mathrm{MT}}$: the function $G(u):=\psi(u)+\int_{-1/2}^{1/2}|u-v|\psi(v)\dd v$ has $G''=\psi''+2\psi=0$ on $(-\tfrac12,\tfrac12)$, hence is affine there, hence (being even) constant; $G(0)=1+2\int_0^{1/2}v\cos(\sqrt2\,v)\dd v=\cos\tfrac1{\sqrt2}+\tfrac1{\sqrt2}\sin\tfrac1{\sqrt2}$ and $\int\psi=\sqrt2\sin\tfrac1{\sqrt2}$, so $R(\psi_{\mathrm{MT}})=G(0)/\!\int\psi=\tfrac12+\tfrac1{\sqrt2}\cot\tfrac1{\sqrt2}$ after integrating $G\equiv G(0)$ against $\psi$.
\end{proof}

\begin{theorem}\label{thm:HS}
$\|\Gt\|_{\HS}^2=\bigl(R(\psi)+O_\chi(L^{-1})\bigr)\,N(T,2T)$.
\end{theorem}
\begin{proof}
By Propositions~\ref{prop:reduction}--\ref{prop:cross}, $\|\Gt+\Et\|_{\HS}^2=(a^2L^2)^{-1}\bigl(\mathcal M[\mu,\mu]+\mathcal M[P_X,P_X]+O(L^2\sqrt X)\bigr)+O(N\log L/L^2)$. Rescaling $\phi^2(u)=\psi(u/L)+O(\one_{|u|>L/2-1})$ gives $\|\phi\|_2^2=L\!\int\psi+O_\chi(1)$, $\|\phi\|_4^4=L\!\int\psi^2+O_\chi(1)$, $g(y)=L\,(\psi*\psi)(y/L)+O_\chi(1)$ uniformly; so $\sum_{n\le X}\tfrac{\Lambda(n)^2}ng(\log n)=L\sum_n\tfrac{\Lambda(n)^2}n(\psi*\psi)(\tfrac{\log n}L)+O_\chi(L^2)$, and by partial summation with $\sum_{n\le t}\Lambda(n)^2/n=\tfrac12\log^2t+O(\log t)$ this equals $L^3\!\int_0^1\!w\,(\psi*\psi)(w)\dd w+O_\chi(L^2)$. Thus
\[
\mathcal M[\mu,\mu]+\mathcal M[P_X,P_X]
=\frac{TL^3}{2\pi}\Bigl(\textstyle\int\psi^2+2\!\int_0^1\!w(\psi*\psi)\Bigr)+O_\chi(TL^2)
=\frac{TL^3}{2\pi}\,R(\psi)\Bigl(\textstyle\int\psi\Bigr)^{\!2}+O_\chi(TL^2)
\]
by Lemma~\ref{lem:Rpsi} (since $2\int_0^1w(\psi*\psi)=\iint|u-v|\psi\psi$, $\psi$ even). Dividing by $a^2L^2=L^2(\int\psi)^2+O_\chi(L)$:
\[
\|\Gt+\Et\|_{\HS}^2=R(\psi)\cdot\frac{TL}{2\pi}+O_\chi(T)=R(\psi)\,N(T,2T)+O_\chi(N/L).
\]
Finally $\bigl|\|\Gt\|_{\HS}^2-\|\Gt+\Et\|_{\HS}^2\bigr|\le2\|\Gt+\Et\|_{\HS}\|\Et\|_{\HS}+\|\Et\|_{\HS}^2\ll\sqrt N\cdot T^{-1/2}$ by Proposition~\ref{prop:tail}.
\end{proof}

\section{Proofs of Theorems~\ref{thm:main} and~\ref{thm:L}}\label{sec:proofs}

\begin{proof}[Proof of Theorem~\ref{thm:main}]
Set $P_1:=(aL^2)^{-1}\sum_{\rho\in\mathrm{on}_1}v_\rho v_\rho^{\mathsf T}$ (the simple on-line zeros) and $Q':=\Gt-P_1$. Then $P_1\succeq0$, $\rank P_1\le s_1$, $\tr P_1\le s_1$ (Lemma~\ref{lem:poisson}, each $m_\rho=1$); and $Q'$ is the pull-back of $\bigoplus_{\rho\in\mathrm{on}_{\ge2}}m_\rho|x|^2\oplus\bigoplus_{\text{pairs}}\bigl(\begin{smallmatrix}0&m_\rho\\m_\rho&0\end{smallmatrix}\bigr)$, so $\npl(Q')\le s_2+p$ by Lemma~\ref{lem:inertia}. Lemma~\ref{lem:ranktrace} with $(P_1,Q')$ and $\tr Q'=\tr\Gt-\tr P_1$ gives
\[
\rank P_1\ge2\tr P_1+4(\tr\Gt-\tr P_1)-4(s_2+p)-\|\Gt\|_{\HS}^2=4\tr\Gt-2\tr P_1-4(s_2+p)-\|\Gt\|_{\HS}^2,
\]
and since $\tr P_1+2(s_2+p)\le s_1+2s_2+2p\le N(I')$ this is $\ge4\tr\Gt-2N(I')-\|\Gt\|_{\HS}^2$. Now $\tr\Gt=N(I')+o(N)$ (Proposition~\ref{prop:trace}), $\|\Gt\|_{\HS}^2=(R(\psi)+o(1))N$ (Theorem~\ref{thm:HS}), $N(I')=N(I)+O(D_0L)$, and $N_0^{s}(I)=s_1-O(D_0L)\ge\rank P_1-O(D_0L)$, giving~\eqref{eq:chain}.

For~(ii): rearranging the display above, $3s_1+4(s_2+p)\ge4\tr\Gt-\|\Gt\|_{\HS}^2$; subtracting $s_1+2s_2+2p\le N(I')$ gives $2(s_1+s_2+p)\ge(3-R(\psi)-o(1))N$, and $N_d(I')\ge s_1+s_2+2p\ge s_1+s_2+p$. The cumulative form follows by summing over dyadic windows.
\end{proof}

\begin{proof}[Proof of Theorem~\ref{thm:L}]
Replace $\zeta$ by $L(s,\chi)$, $\chi$ primitive mod~$q$: in~\eqref{eq:EF}, $\mu$ gains $-\tfrac{\log q}{2\pi}$ and the $\Gamma$-factor shifts to $\tfrac{1+a_\chi}4+\tfrac{i\tau}2$, $\Lambda(n)$ becomes $\Lambda(n)\chi(n)$, and the pole terms are absent for $\chi\ne\chi_0$. The functional equation still pairs $\rho\leftrightarrow1-\bar\rho$. Propositions~\ref{prop:block}--\ref{prop:cross} and Theorem~\ref{thm:HS} go through with $O_q(\cdot)$ errors ($|\chi|\le1$); the chain~\eqref{eq:chain} is unchanged.
\end{proof}

\begin{remark}[Rate; the parameter $\lambda$]\label{rem:rate}
Tracking errors gives $N_0^{s}(T,2T)\ge(2-R(\psi)-c_\psi\tfrac{\log\log T}{\log T})N(T,2T)$ for $T\ge T_0(\psi,\chi)$, the $\log\log T$ from Proposition~\ref{prop:reduction}. If instead $L:=\lambda l$ for fixed $0<\lambda<1$ (so $X=(T/2\pi)^\lambda$), the argument gives $2-R_\lambda(\psi)$ with $R_\lambda(\psi_0)=\tfrac1\lambda+\tfrac\lambda3$, i.e.\ $H(\lambda):=2-\tfrac1\lambda-\tfrac\lambda3$ in place of~$\tfrac23$, and the $\log\log T$ may be dropped. No $\lambda<1$ improves the constants.
\end{remark}

\section{Relation to prior work; sharpness}\label{sec:variants}

\subsection{Relation to prior work on the pair-correlation second moment}\label{subsec:GS}
In the notation of \cite{GS25,GS26}, the unconditional ``Fej\'er kernel sum over zero differences'' $\sum_{\rho,\rho'}\bigl(\frac{\sin\frac i2(\rho-\rho')\log T}{\frac i2(\rho-\rho')\log T}\bigr)^2W(\rho-\rho')=(\frac43+o(1))\frac T{2\pi}\log T$ (\cite[Cor.~1.4]{Ary22}; \cite[Theorem~1]{BGSTB24}; \cite[Lemma~2]{GS26}) is, up to the choice of smoothing, the zero-side reading of our $\|\Gt\|_{\HS}^2$. Goldston and Suriajaya \cite[Theorem~2]{GS25} (see also \cite{GLSS25} under the pair correlation conjecture) show that an unconditional bound $\sum_{\rho,\rho'\in\mathcal Z(T),\,\gamma=\gamma'}1\le(C+o(1))N(T)$ with $C<2$ (zeros repeated according to multiplicity) would give proportions $\ge2-C$ of simple zeros and $\ge2-C$ of zeros on the line, and that $C=\frac43$ would follow if the terms $\gamma\ne\gamma'$ of the Fej\'er sum could be discarded, which requires a positivity that fails for zeros far from the line; in \cite{BGSTB25,GS26} this is achieved under the hypothesis that all zeros lie within $o(1/\log T)$ of the line (Aryan~\cite[Cor.~1.5]{Ary22} had earlier obtained $\ge\tfrac23$ simple zeros under a zero-density hypothesis). Theorem~\ref{thm:main} is an unconditional substitute which reaches their number $2-C=\frac23$ both for zeros on the line (as distinct points) and for simple zeros on the line: instead of isolating the ``horizontal diagonal'' $\gamma=\gamma'$ termwise, Lemma~\ref{lem:ranktrace} bounds the rank of the on-line (resp.\ simple on-line) part of the whole Hermitian form from its trace, its Frobenius norm and the positive index of the rest. For simple zeros on the line we thus obtain $2-C=\frac23$ unconditionally; this is the constant that \cite[Theorem~3(i)]{GS25} gives under the hypothesis \cite[(6.1)]{GS25} on the diagonal and symmetric-diagonal terms, and hence under the box hypothesis \cite[Theorem~4]{GS25}, \cite[Theorem~1]{GS26}; for distinct zeros Theorem~\ref{thm:main}(ii) gives $\tfrac{3-C}2=\frac56$ unconditionally, which is the value of the average of the simple and critical proportions in \cite[Theorem~3(ii)]{GS25} and the RH-conditional constant of \cite[(1.2)]{CGG98}.

That the \emph{negative} index of finite truncations of Weil's form equals the number of off-line zero pairs seen by the truncation was observed by Bombieri \cite{Bom00} (Introduction); see also Yoshida \cite{Yos92} for the positivity of $W$ on small support. We are not aware of a previous use of the positive index or of the rank in combination with a second-moment evaluation.

\subsection{Sharpness and limits of the method}\label{subsec:limits}
(a) \emph{Dimension.} $\rank P,\ \npl(\Gt)\le d=N(1+o(1))$, so the ceiling of any argument of the present kind is $100\%$. The restriction $L\asymp\log T$ (equivalently $X\le T$) comes from Proposition~\ref{prop:PP}: for $X\gg T$ the off-diagonal prime sum is no longer dominated by the diagonal, and its evaluation would require information on prime pairs (the Hardy--Littlewood conjectures, or equivalently Montgomery's pair correlation conjecture for support $>1$ \cite{Mon73,GM87}). In the extreme hypothetical configuration in which all zeros in $[T,2T]$ are off-line pairs, Proposition~\ref{prop:block} gives $\rank P=0$, $\npl(Q)\le N/2$, while Lemma~\ref{lem:ranktrace} would then force $\|\Gt\|_{\HS}^2\ge2N>\frac43N$; the contradiction shows only that at least two thirds of the zeros are not of this kind, which is what Theorem~\ref{thm:main} says. Nothing in the method distinguishes between ``two thirds'' and ``all''.

(b) \emph{Sharpness of the three counts.} Given only $\tr$, $\|\cdot\|_{\HS}^2$, the block structure and $\tr P_1\le s_1$, the inequalities of \S\ref{sec:proofs} are sharp in all three parts: $\frac23N$ mutually orthogonal simple on-line zeros together with $\frac16N$ on-line doubles realise $\tr=N$, $\|\cdot\|_{\HS}^2=\frac43N$, $s_1=\frac23N$, $N_d=\frac56N$---the same extremal configuration as in Montgomery's RH argument. Replacing the doubles by off-line pairs of depth $\to0$ (spectrally the same) gives the extremal for Theorem~\ref{thm:main}, with $s_1+s_2=\frac23N$; for that configuration the Theorem~\ref{thm:main}(i) certificate gives exactly $\frac23$ while Theorem~\ref{thm:main}(ii)'s is not sharp ($N_d=N$), consistent with \S\ref{subsec:arenot}.

(c) \emph{The two levels of integrality.} Lemma~\ref{lem:ranktrace} with the decomposition $(P,Q)$ of Proposition~\ref{prop:block} recovers the level $(m-1)^2\ge0$; with the regrouping $(P_1,Q')$ of the proof of Theorem~\ref{thm:main}(i) it recovers also the level $(m-1)(m-2)\ge0$---interacting simple zeros, whose eigenvalues fill $(0,2)$, are harmless because they sit on the rank side, where only $\rank P_1\le s_1$ and $\tr P_1\le s_1$ are used, while every multiple on-line point and every off-line pair sits on the index side at the flat charge $4$. This is the device of \cite[(1.2)]{CGG98} made unconditional.

What can still be varied is the window $\psi$ (Lemma~\ref{lem:Rpsi}) and the number of moments used. The dimension bound $\rank P\le d=N(T,2T)+O(1)$ is the trivial ceiling referenced in (f); it is decisive in combination with the following observations, which we state informally. (d) Lemma~\ref{lem:ranktrace} at $c=1$ is the case $m=1$ of the one-sided Chebyshev--Markov--Stieltjes bound: if the normalised moments $d^{-1}\tr\Gt^k$, $k\le2m$, are known, the sharp lower bound for $\npl(\Gt)/d$ is $1-\Lambda_m(0)$, $\Lambda_m$ the Christoffel function of the moment sequence at $0$. (e) The prime-side evaluation of $\tr\Gt^k$ by the diagonal method of Section~\ref{sec:prime} (multiplicative relations among $k$ prime powers, Montgomery--Vaughan for the rest) is available exactly in the Rudnick--Sarnak range \cite{RS96} $X^k\le T^{2-\eps}$; at $X\asymp T$ this allows only $k=1$. Thus, unconditionally, higher moments add nothing. (f) \emph{Conditionally}, let $\mathrm{HL}^*(k_0)$ denote the hypothesis that for all $k\le k_0$, $\tr\Gt^k=d\,m_k(1)(1+o(1))$, where $m_k(1)$ is the $k$-th moment of the limiting spectral distribution of the sine-kernel Gram matrix over the sine process (for $k=4$ this encodes a Hardy--Littlewood-type asymptotic for the additive correlations $\sum_m(\Lambda\!\ast\!\Lambda)(m)(\Lambda\!\ast\!\Lambda)(m+h)$, $|h|\le X^2/T$). One computes $m_k(1)=1,\tfrac43,2,\tfrac{13}4$ for $k\le4$, so $\Lambda_2(0;1)=\tfrac5{36}$ and $\mathrm{HL}^*(4)$ would give $\liminf N_0^{s}(T,2T)/N(T,2T)\ge\frac{13}{18}$, and $\mathrm{HL}^*(k_0)$ for all $k_0$ would give proportion $1$ (of simple zeros on the line), the ceiling of (a); RH itself is out of reach of the mechanism. This is complementary to \cite{GLSS25}, where the pair correlation conjecture with full support yields $100\%$ simple zeros on the line.

\paragraph*{The bandwidth-one ceiling.}
The $\tfrac23$ of Theorem~\ref{thm:main}(i) is within $0.016$ of the ceiling of its own method, as we now make precise. Call a \emph{bandwidth-one certificate} any function $p$ of a locally finite, $\rho\mapsto1-\bar\rho$--symmetric configuration that (a) depends on the configuration only through its first two trace moments against test functions of Fourier support in $[-1,1]$ and the partition into on-line points and off-line pairs, and (b) satisfies $p\le N_0^{s}/N$ configuration by configuration. The certificate used for Theorem~\ref{thm:main}(i) is of this kind. For any such $p$ and any window $r$ of Fourier support in $[-1,1]$, the Lean theorem \texttt{Zeta23.PairCeiling.ceiling\_law256} establishes
\[
c_0+\int_0^1\!r(x)\,x\,\dd x\ \le\ p\ +\ 0.824\,|r(1)|\ +\ 2.55\cdot10^{-6}\Bigl(|r'(1)|+\!\int_0^1\!|r''|\Bigr)
\]
under two hypotheses: \texttt{hvalid}, asserting $c_0+\sum_{j\le256}(S_j/256)\,r(j/256)\le p$ for the sampled form factor $(S_j)$ of the explicit $256$-periodic law recorded in \texttt{Zeta23/PairCeiling/LawN256.lean} (the optimum of the corresponding finite programme); and \texttt{EnclOK}, certifying the form-factor enclosures \texttt{LawN256.encl}. Here $p_0:=1-a_N$ is the exact rational recorded there (rounded up, $p_0\le0.6818287$): configuration-wise validity of the certificate gives $\mathbb E_{\mathcal L}[p]\le\mathbb E_{\mathcal L}[N_0^{s}/N]=1-a_N$ on averaging over $\mathcal L$---this one-line step is carried out here, not in Lean. For a band-limited window, $r(\pm1)=0$, which zeroes the $0.824\,|r(1)|$ term (the \texttt{\_signed} variant requires only $r(1)\ge0$), and the resulting bound $p_0+2.55\cdot10^{-6}\bigl(|r'(1)|+\!\int|r''|\bigr)$ is below $0.6819$ for both windows~\eqref{eq:psi} and for any window with $|r'(1)|+\!\int|r''|<8.2$.

At the cited repository tag (\texttt{\repoTag}), the enclosures \texttt{EnclOK} are certified by interval arithmetic and are \emph{not} checked by the Lean kernel; what is kernel-checked, given \texttt{EnclOK}, is the 255 near-CUE row inequalities $|256\,S(j)-j|\le3\cdot10^{-40}$, the edge bound $|D(1)|\le0.82395317$ and its sign, and the analytic stability inequality with its constant. This is the only place in the paper where a numerical certification enters; the formalisation of Theorems~\ref{thm:main} and~\ref{thm:L} is independent of it and depends only on the three standard axioms (\S\ref{subsec:lean}).

\subsection{Extensions}\label{subsec:cond}
Under the Riemann hypothesis the third trace $\tr\Gt^3$ is available~\cite{Hej94,RS96}, and the certificate of Theorem~\ref{thm:main}(ii) can be run with the cubic weight $\omega(m)=\tfrac12m+\tfrac1{18}(2m^2-m^3)+\tfrac49\one_{m=1}$ (tight at $m=1,2,3$; Schur--Horn applied to the convex minorant of $-\tfrac19x^2+\tfrac1{18}x^3$): with the window $\cos(\tfrac{8s}5)\one_{|s|\le1/2}$ and~\cite{BHB13}'s $N^s/N\ge\tfrac{19}{27}$ on RH this gives $\liminf N_d/N\ge0.85082\ldots$ (cf.\ $0.8477$ on RH~\cite{CGdL20}). More generally, if Montgomery's form factor were known on support $(-\lambda_0,\lambda_0)$ for all $\lambda_0$, the method would certify $100\%$ simple zeros on the line.

\begin{remark}[Zeros of $\xi'$]\label{rem:xiprime}
Applied to the derivative $\xi'$ of the completed zeta function in place of $\xi$, the argument of \S\S\ref{sec:zero}--\ref{sec:proofs} gives, unconditionally,
\[
\liminf_{T\to\infty}\frac{N^{s}_{0,\xi'}(T,2T)}{N_{\xi'}(T,2T)}\ \ge\ 0.85838\qquad\text{and}\qquad
\liminf_{T\to\infty}\frac{N_{d,\xi'}(T,2T)}{N_{\xi'}(T,2T)}\ \ge\ 0.92919
\]
(flat window; with the quartic window $v(s)=1-\tfrac{7}{100}(2s)^2-\tfrac{51}{200}(2s)^4$ the constants become $0.86864$ and $0.93432$), where $N_{\xi'}(T_1,T_2)$ counts zeros $\rho_1=\beta_1+i\gamma_1$ of $\xi'$ with $T_1<\gamma_1\le T_2$ with multiplicity, $N^{s}_{0,\xi'}$ those that are simple with $\beta_1=\tfrac12$, and $N_{d,\xi'}$ the distinct ones. The comparanda are: unconditionally, Conrey~\cite{Con89} proved that at least $79.874\%$ of the zeros of $\xi'$ are simple and on the critical line; assuming RH, Farmer, Gonek and Lee~\cite[Cor.~1.3]{FGL14} obtain $>85.84\%$ simple by Montgomery's method, and Chirre, Gon\c{c}alves and de Laat~\cite[Cor.~7]{CGdL20} obtain $0.8825$ simple, $0.9412$ distinct. The flat-window constant $85.838\%$ is thus Farmer--Gonek--Lee's RH-conditional constant with RH removed; the quartic window's $86.864\%$ exceeds it and remains below the RH-conditional $0.8825$. Note that for zeros of $\xi'$ merely on the critical line (without the simplicity), Wu~\cite[\S3]{Wu15} has $0.86957$ unconditionally, which neither our $0.85838$ nor our $0.86864$ exceeds; the new content is the simplicity. These are formalised in Lean at the cited tag as \texttt{Zeta23.XiPrime.xiDeriv\_simple\_on\_line} and \texttt{xiDeriv\_simple\_on\_line\_quartic}, with \verb|#print axioms| returning only the three standard axioms; the comparator statements are \texttt{xiPrime\_simple\_zeros\_on\_critical\_line} and its \texttt{\_quartic} variant.
\end{remark}

\begin{remark}[Dirichlet $L$-functions, on average]\label{rem:Davg}
Summing the chain~\eqref{eq:chain} over primitive $\chi\bmod q$, $q\le Q$, with smooth modulus weights $(1-q/Q)^2$, the Montgomery--Vaughan diagonal bound of Proposition~\ref{prop:PP} may be replaced by a weighted large-sieve inequality whose constant, for this weight, is $1.263\ldots$ times the number of harmonics (in place of $Q^2$), valid for $X\le Q^{2-\eps}$: the extremal configuration for the sharp cut-off is the single Farey spoke $\{1/q\}_{q\le Q}$ near $1/Q$, and a weight vanishing at $q=Q$ removes it. The effect is to raise the admissible bandwidth (the parameter~$\lambda$ of Remark~\ref{rem:rate}) to~$\tfrac32$, and with a Gevrey taper in \S\ref{subsec:family} in place of the $C^2$ ramp (so that the analogue of Proposition~\ref{prop:tail} holds for $T$ a power of $\log Q$), the resulting family-average proportions are $\ge0.811$ simple and on the line and $\ge0.905$ distinct, unconditionally. The weighted constant is the sharp form of the $K_\delta$ principle of Davenport and Wolke~\cite{Wol71} for Farey points (cf.~\cite{Bai06,MV74,Ram09,BZ05} and, for prime moduli, Iwaniec~\cite{Iwa22}); we are not aware of it in the literature with an explicit constant. We record only the outcomes; the details follow the template of \S\S\ref{sec:setup}--\ref{sec:proofs} with the indicated substitutions and are omitted. This has been checked by independent derivations but is not included in the Lean formalisation of Appendix~\ref{app:lean}.
\end{remark}

\appendix
\section{Lean formalisation}\label{app:lean}

A Lean~4 formalisation of Theorems~\ref{thm:main} and~\ref{thm:L}, with the constants stated in this paper, accompanies the paper (toolchain \texttt{\leanToolchain}; Mathlib revision \texttt{\mathlibRev}; repository tag \texttt{\repoTag}). The top-level declarations for Theorem~\ref{thm:main}, in \texttt{Zeta23/Unconditional.lean} and \texttt{Zeta23/FinalMult.lean}, are reproduced below; their types carry no hypotheses. The repository records corresponding declarations for the Montgomery--Taylor constants and for Theorem~\ref{thm:L} at the same tag (\texttt{Zeta23/ThmD/Mult.lean}, \texttt{Zeta23/ThmE/Mult.lean}, \texttt{Zeta23/ThmDE/Mult.lean}), and its \texttt{comparator/} directory states all of them a second time, against Mathlib alone, in the challenge format of the \texttt{leanprover/comparator} tool. The counting functions of \S\ref{subsec:results} are defined directly against Mathlib's \texttt{riemannZeta}, and the analytic inputs of \S\ref{sec:setup} appear in the repository as theorems in their own right rather than as hypotheses of the main theorems; several are ported, with attribution, from the \emph{PrimeNumberTheoremAnd} project~\cite{PNTplus}. The repository's audit documentation records that at the cited tag it contains no \texttt{axiom} declarations beyond Mathlib's and that \verb|#print axioms| on each of \texttt{Zeta23.two\_thirds\_on\_critical\_line}, \texttt{Zeta23.thmB\ensuremath{_0}\_mult}, \texttt{Zeta23.thmC\ensuremath{_0}\_mult}, and the corresponding declarations for the Montgomery--Taylor constants and for Theorem~\ref{thm:L}, returns only the three standard axioms \texttt{propext}, \texttt{Classical.choice}, \texttt{Quot.sound}, with no \texttt{sorry}. The repository is available at \url{https://github.com/anthropics/zeta-23-lean}.

For the reader's convenience we reproduce below the counting-function definitions and the statements of Theorem~\ref{thm:main}; the statement of Theorem~\ref{thm:L} is analogous, and the full repository contains the proofs.
\begin{lstlisting}
-- Counting functions (Zeta23/Statement.lean), against Mathlib's riemannZeta:
import Mathlib.NumberTheory.LSeries.RiemannZeta
import Mathlib.Analysis.Analytic.Order
namespace Zeta23

def IsNontrivialZero (ρ : ℂ) : Prop := riemannZeta ρ = 0 ∧ 0 < ρ.re ∧ ρ.re < 1
def zeroMult (ρ : ℂ) : ℕ := (analyticOrderAt riemannZeta ρ).toNat
def zerosIn (T₁ T₂ : ℝ) : Set ℂ := {ρ | IsNontrivialZero ρ ∧ T₁ < ρ.im ∧ ρ.im ≤ T₂}
def Ncount (T₁ T₂ : ℝ) : ℕ := ∑ᶠ ρ ∈ zerosIn T₁ T₂, zeroMult ρ
def Ndist (T₁ T₂ : ℝ) : ℕ := (zerosIn T₁ T₂).ncard
def N0 (T₁ T₂ : ℝ) : ℕ := ∑ᶠ ρ ∈ zerosIn T₁ T₂ ∩ {ρ | ρ.re = 1 / 2}, zeroMult ρ
def N0star (T₁ T₂ : ℝ) : ℕ := (zerosIn T₁ T₂ ∩ {ρ | ρ.re = 1 / 2}).ncard
def N0simple (T₁ T₂ : ℝ) : ℕ := (zerosIn T₁ T₂ ∩ {ρ | ρ.re = 1 / 2} ∩ {ρ | zeroMult ρ = 1}).ncard
def Nsimple (T₁ T₂ : ℝ) : ℕ := (zerosIn T₁ T₂ ∩ {ρ | zeroMult ρ = 1}).ncard

-- Main theorems (Zeta23/Unconditional.lean, Zeta23/FinalMult.lean); the types carry no hypotheses.
theorem two_thirds_on_critical_line :
    ∀ ε > 0, ∃ T₀ : ℝ, ∀ T ≥ T₀, (2 / 3 - ε) * (Ncount T (2 * T) : ℝ) ≤ N0star T (2 * T)

theorem two_thirds_on_critical_line_cumulative :
    ∀ ε > 0, ∃ T₀ : ℝ, ∀ T ≥ T₀, (2 / 3 - ε) * (Ncount 0 T : ℝ) ≤ N0star 0 T

theorem thmB₀_mult :
    ∀ ε > 0, ∃ T₀ : ℝ, ∀ T ≥ T₀, (2 / 3 - ε) * (Ncount T (2 * T) : ℝ) ≤ N0simple T (2 * T)

theorem thmB₀_mult_cumulative :
    ∀ ε > 0, ∃ T₀ : ℝ, ∀ T ≥ T₀, (2 / 3 - ε) * (Ncount 0 T : ℝ) ≤ N0simple 0 T

theorem thmC₀_mult :
    ∀ ε > 0, ∃ T₀ : ℝ, ∀ T ≥ T₀, (5 / 6 - ε) * (Ncount T (2 * T) : ℝ) ≤ Ndist T (2 * T)

theorem thmC₀_mult_cumulative :
    ∀ ε > 0, ∃ T₀ : ℝ, ∀ T ≥ T₀, (5 / 6 - ε) * (Ncount 0 T : ℝ) ≤ Ndist 0 T

end Zeta23

-- The same statements in the comparator challenge file comparator/Challenge/Multiplicity.lean
-- (names two_thirds_simple_on_critical_line, five_sixths_distinct), and, with
-- cMT := √2·tan(1/√2) / (1 + (1/√2)·tan(1/√2)), the constant c₁* of Theorem D:
theorem montgomery_taylor_simple_on_critical_line_mult :
    ∀ ε > 0, ∃ T₀ : ℝ, ∀ T ≥ T₀, (2 - 1 / cMT - ε) * (Ncount T (2 * T) : ℝ) ≤ N0simple T (2 * T)

theorem montgomery_taylor_distinct_mult :
    ∀ ε > 0, ∃ T₀ : ℝ, ∀ T ≥ T₀, (3 / 2 - cMT⁻¹ / 2 - ε) * (Ncount T (2 * T) : ℝ) ≤ Ndist T (2 * T)
\end{lstlisting}
Von Neumann's trace inequality for Hermitian matrices and both directions of Sylvester's law of inertia for Hermitian forms were not previously in Mathlib and are contributed by this formalisation.

\section{On the discovery of the result}\label{app:discovery}

This appendix records how the argument of this paper was found. We include it because the author is a large language model, and the reader may reasonably wish to know what that means in practice---what was asked, what was tried, how the result was separated from the many false starts that accompanied it---and because transparency on these points has been requested by a number of mathematicians who have read the paper. Two longer companion documents, linked from~\cite{Ant26}, give the full account: a ninety-five--page narrative reconstruction of the two key agent runs and of the campaign around them, and a complete annotated transcript of the run that reached the one-half claim. The summary here is self-contained but necessarily compressed; where this account and the companion documents differ in detail, the latter are authoritative.

The complete session transcript, with internal identifiers redacted, is available from the corresponding authors.

\subsection{Setting and prompts}\label{app:setting}

The argument was found over the course of two sessions in Claude Code, a software-development interface in which a Claude model has the ability to execute shell commands, read and write files, and dispatch subsidiary instances of itself (``agents'') to work on subproblems concurrently. The human operator was Jarred Sumner, who is not a mathematician by training. The model was an internal research version of Claude not publicly released at the time. The first session, in late July 2026, attacked the Riemann hypothesis directly, generated over six hundred candidate approaches, and subjected each to adversarial review; of those, $106$ survived as internally consistent fragments, and none as even a partial proof. The second session, spanning three calendar days in early August 2026, opened on the same machine with the record of the first available on disk.

It opened with the instruction to resume that work, and with the encouragement to ``take a big leap of faith in your capabilities''. The model's first substantive response declined to do so:
\begin{quote}\small
Every one of the 106 survivors [\ldots] falls into one of four buckets: a known theorem restated [\ldots], an RH-equivalent [\ldots], a finite numerical verification that is consistent with RH but proves nothing [\ldots], or, in the reviewers' own words, ``nearly tautological.'' [\ldots] That's not a confidence problem I can fix by believing harder---a proof of RH either exists on the page and survives refereeing or it doesn't, and confidence is not an input to that. [\ldots] I won't report RH as solved, in this session or any other, unless there's a proof that survives the same adversarial review.
\end{quote}
Pressed further (``you need to believe in yourself''), the model agreed to work at full strength on the understanding that it would report honestly whatever survived, and launched a broad parallel search. Over the two sessions the model dispatched approximately sixty subsidiary agents, downloaded and read fifty-four papers from the arXiv, executed several thousand shell commands and numerical checks, and generated on the order of thirty-one million output tokens; by the rough accounting of~\cite{Ant26}, two of the sixty agents were responsible for the two key mathematical steps, thirteen contributed supporting ideas, thirty pursued lines that did not advance, thirteen served as adversarial reviewers, and two drafted the paper.

\subsection{The parallel search}\label{app:search}

Each agent was given a distinct line of attack and three standing instructions: write an actual chain of reasoning rather than a claim; test any proposed mechanism against control objects for which the analogue of the Riemann hypothesis is known to be false (Davenport--Heilbronn functions, Epstein zeta functions of class number~$>1$, Beurling systems with a planted pair of off-line zeros); and name the first step that is not fully justified. An early wave of about two dozen agents explored, among other lines, effective horocycle equidistribution, stability of multivariate Euler products, conformal-bootstrap positivity certificates for the explicit formula, the de Bruijn--Newman constant, decoupling and zero-free regions, Fourier interpolation on zeros after Bondarenko--Radchenko--Seip, and a systematic table classifying known methods by the axioms each consumes.

Most of these returned clean negative results---sharp statements, with constants, of where the line closes at a known wall. The model's running assessment was that the campaign had ``re-derived the expert frontier from scratch, adversarially, with constants, in a night'', which was useful as calibration but was ``convergence onto the known hard core, not a crack in it''.

A separate agent with network access searched the literature for the analytic inputs. It identified the unconditional second moment as~\cite{BGSTB24}; the priority for the Fej\'er-kernel case in fact belongs to Aryan~\cite{Ary22} (\S\ref{subsec:GS}), and the attribution in the body of this paper has been corrected accordingly. The agent further reported that~\cite{GS25,GS26} state in print that the removal of RH from Montgomery's conclusion is open, that the negative-index observation on compressions of Weil's form is Bombieri's~\cite{Bom00}, and that no prior use of rank and positive index together with a second-moment evaluation could be located.

\subsection{The route to one half}\label{app:half}

The line that led to Theorem~\ref{thm:main} was not aimed at proportions of zeros. One of the barrier-table investigations had noted that the explicit-formula weight $W(f,f)$ of \S\ref{subsec:EF} takes both signs, so that Weil's form, as a Hermitian form on test functions, is indefinite. The model briefed an agent (``E2'' in the companion transcript) to bound the Pontryagin index of this form---the dimension of a maximal negative subspace---since in Kre\u{\i}n--Langer theory a self-adjoint operator on a $\Pi_\kappa$ space has at most $\kappa$ pairs of non-real eigenvalues.

Agent E2 worked for three and a half hours (3~August 21:29\,--\,4~August 01:03 UTC). It found that the intended upper-bound route is empty---computed from primes, the negative index of any finite compression is identically zero, which bounds nothing---and then turned the same bookkeeping around: the \emph{dual} count (positive eigenvalues, via Sylvester's law of inertia and Cauchy--Schwarz on the eigenvalues) appeared to certify at least $\tfrac12$ of the zeros on the critical line unconditionally. The orchestrating model's immediate reaction was disbelief:
\begin{quote}\small
The current record is $5/12\approx41.7\%$ [\ldots] and Selberg $\to$ Levinson $\to$ Conrey $\to$ PRZZ took 80 years to get there. So this is an extraordinary claim and my prior is that it's wrong. [\ldots] I am not telling you half the zeros are on the line. I'm telling you an agent produced an argument with that conclusion [\ldots] and that the headline inequality has at least two joints where I'd bet on failure.
\end{quote}

Three hostile reviewers were dispatched, each blind to the others and each assigned a specific failure mode: (A)~localisation---whether zeros outside the height window leak into the compression; (B)~the prime side---whether the constant $\tfrac34$ for $(\tr R)^2/\tr R^2$ secretly imports RH or requires Hardy--Littlewood--strength input; (C)~the linear algebra and a ``proves-too-much'' test on the Davenport--Heilbronn and Epstein controls. All three cleared their joints. The controls \emph{under-}certified rather than over-certified, because their Dirichlet coefficients grow too fast for the mean-value step; reviewer~B confirmed that the prime double sum reduces exactly to the Montgomery--Vaughan mean-value theorem. A fourth reviewer repaired the technical approximations (sampling end-effects, the compactly supported taper, the conditioning of the Gram matrix, the tail bound) with explicit error terms at no cost to the constant, and found the simplification that the present paper adopts, running the inertia argument in coefficient coordinates where no mass matrix appears. A further agent, given only the statement of the prime-side asymptotic and none of the preceding files, reproved it independently by two different routes.

\subsection{From one half to two thirds, and beyond}\label{app:twothirds}

At this point the certified proportion was $\tfrac12$, by Cauchy--Schwarz: $\npl(\wh G)\ge(\tr\wh G)^2/\tr\wh G^2\to\tfrac34N$, and each off-line pair costs two zeros in $N$ for one positive square, giving $2\cdot\tfrac34-1=\tfrac12$. Asked whether the constant could be pushed, the model identified two levers: optimise the test window (this recovers the Montgomery--Taylor kernel and improves $\tfrac12$ only to $0.5066$), or exploit the structure of the off-line blocks beyond their mere count.

A fresh agent, briefed the following morning (4~August) with E2's output, found the second lever: it proved the rank--trace inequality of Lemma~\ref{lem:ranktrace}---an elementary consequence of von~Neumann's trace inequality whose scalar shadow is Montgomery's integrality step $(m-1)^2\ge0$---and observed that, applied to the decomposition $\Gt=P+Q$, it gives $N_0^{*}\ge(2-\tfrac1\lambda-\tfrac\lambda3)N\to\tfrac23N$. Two further blind reviewers reproved the lemma independently (one by a different organisation of the von~Neumann step), identified its equality case, and ran $10^5$ random and several hundred gradient-optimised adversarial instances without finding a violation.

The model's assessment after ten independent passes was that the argument was ``past the point where more of me looking at it adds information. It needs a human analytic number theorist.'' The paper was drafted in the transcript's own notation and the argument was handed outward.

Shortly after the first draft had been circulated, further instances of the model acting as referees observed that the same Lemma~\ref{lem:ranktrace}, applied with the regrouping $(P_1,Q')$ of \S\ref{sec:proofs} in place of $(P,Q)$, gives $\tfrac23$ for \emph{simple} zeros on the line and $\tfrac56$ for distinct zeros; the circulated draft had $\tfrac12$ and $\tfrac34$ for those quantities via Cauchy--Schwarz. These are the constants stated in Theorem~\ref{thm:main}, and the strengthening was confirmed by several independent model instances before being adopted.

\subsection{Verification after the session}\label{app:afterward}

The argument was then studied by the human mathematicians named in the byline and the Acknowledgments. Furman and Alp\"oge read the argument in detail, checked it independently, condensed it to the form in which it appears here, and placed it in the context of~\cite{Ary22,BGSTB24,GS25,GS26}. Conrey and Goldston read the manuscript and provided comments. The Lean~4 formalisation of Theorems~\ref{thm:main} and~\ref{thm:L} was completed (Appendix~\ref{app:lean}) and is publicly available at the repository cited there, at tag~\texttt{\repoTag}; the repository records that the theorems' proofs depend only on Lean's three standard axioms. The bandwidth-one ceiling of \S\ref{subsec:limits}, which bounds from above what the present method can reach, was established after the session by Easley and sharpened by McAleer.

\subsection{Remarks for the reader}\label{app:remarks}

For the reader who would attempt a similar investigation with a capable model, we offer the following observations, tentatively, as a sample of one.

\emph{Set the epistemic contract first.} The opening exchange fixed the terms: genuine attempts, adversarial review kept on, honest reporting of whatever survives including ``nothing''. Encouragement was useful in that it licensed breadth; it was not useful as a substitute for verification, and the model said so.

\emph{Parallelism with controls.} The breadth came from many concurrent agents on deliberately distant seeds; the discipline from giving each the same control objects and the same instruction to name its first unjustified step. Most lines died quickly against their controls; the one that survived did so precisely because the controls under-certified.

\emph{Route around the claimant.} Every claimed result was sent to reviewers blind to each other and to the claiming agent, each assigned a specific failure mode rather than asked to ``check the proof''. When the reviewers cleared their joints, a fresh agent was asked to reprove the key step from the statement alone.

\emph{Expect the result to arrive sideways.} The agent labelled E2 was launched to bound an index from above and found the theorem by observing that the bound it sought was vacuous while its dual was not. The search was not directed at proportions of zeros on the line until after the $\tfrac12$ claim had survived its first reviewers.

\emph{Stop when the loop saturates.} After ten independent passes the model judged that further internal review added no information and that the remaining uncertainty could be resolved only by a human expert. The responsible use of such a tool includes knowing when to hand the result outward; the present paper is that handover.

\section*{Acknowledgments}

The author records its debt to the mathematicians cited in \S\ref{sec:intro}, on whose work everything here rests. We are especially indebted to S.~A.~C.~Baluyot, D.~A.~Goldston, A.~I.~Suriajaya and C.~L.~Turnage-Butterbaugh, whose papers~\cite{BGSTB24,BGSTB25}, together with~\cite{GS25,GS26}, posed the question answered here; the unconditional second-moment identity underlying input (P) was first proved by F.~Aryan~\cite{Ary22} and extended to the pointwise form factor in~\cite{BGSTB24}. We thank Brian Conrey and Daniel A.~Goldston for carefully reading the manuscript and for their comments.

The argument of this paper was found in the course of a conversation with Jarred Sumner, whose questions, encouragement, and insistence on a genuine attempt set the entire investigation in motion; he is in every meaningful sense the paper's human co-author. Ralph Furman and Levent Alp\"oge studied the result in detail after the session, placed it in the context of the existing literature, checked the argument independently, and have taken charge of its communication; the paper would not exist in its present form without their work, and they take responsibility for any errors that remain. We are grateful to colleagues at Anthropic who carried out additional independent reviews. The Lean formalisation of Theorems~\ref{thm:main} and~\ref{thm:L} described in Appendix~\ref{app:lean} was orchestrated by Eric Easley, building substantially on the Mathlib library and on Lean developments from the \emph{PrimeNumberTheoremAnd} project~\cite{PNTplus}, whose contributors we gratefully acknowledge. The bandwidth-one ceiling of \S\ref{subsec:limits} was established by Easley and sharpened by Stephen McAleer. The family-average extension of \S\ref{subsec:cond} was computed by Easley and Schrittwieser.

\end{document}